\documentclass[a4paper]{article}

\usepackage[english]{babel}
\usepackage[utf8x]{inputenc}
\usepackage[T1]{fontenc}

\usepackage[a4paper,top=3cm,bottom=2cm,left=3cm,right=3cm,marginparwidth=1.75cm]{geometry}

\usepackage{amsmath, amsfonts, amsthm, amssymb, MnSymbol}
\usepackage{graphicx}
\usepackage[colorinlistoftodos]{todonotes}
\usepackage[colorlinks=true, allcolors=blue]{hyperref}
\usepackage{accents}
\usepackage[margin=1cm]{caption}
\usepackage{ulem}

\newcommand{\R}{\mathbb{R}}

\newcommand{\EE}{\mathbb{E}}

\newcommand{\III}[1]{{\left\vert\kern-0.25ex\left\vert\kern-0.25ex\left\vert #1 
    \right\vert\kern-0.25ex\right\vert\kern-0.25ex\right\vert}}

             \newcommand{\cn}{\mathcal{N}}           

\newcommand{\threequals}{\equiv}

\newcommand{\ed}{\stackrel{d}{=}} 

\newcommand{\argmin}{\textrm{argmin}}

\theoremstyle{plain}
\newtheorem{theorem}{Theorem}
\newtheorem{corollary}[theorem]{Corollary}
\newtheorem{remark}[theorem]{Remark}
\newtheorem{lemma}[theorem]{Lemma}

\makeatletter
\providecommand*{\diff}%
	{\@ifnextchar^{\DIfF}{\DIfF^{}}}
\def\DIfF^#1{%
	\mathop{\mathrm{\mathstrut d}}%
		\nolimits^{#1}\gobblespace}
\def\gobblespace{%
	\futurelet\diffarg\opspace}
\def\opspace{%
	\let\DiffSpace\!%
	\ifx\diffarg(%
		\let\DiffSpace\relax
	\else
		\ifx\diffarg[%
			\let\DiffSpace\relax
	\else
		\ifx\diffarg\{%
			\let\DiffSpace\relax
		\fi\fi\fi\DiffSpace}

\newcommand{\cm}{\mathcal{M}} 

\newcounter{rcnt}[section]

\def\qt#1{\qquad\text{#1}}

\numberwithin{equation}{section}
\numberwithin{theorem}{section}

\begin{document}

\title{Distribution-free properties of isotonic regression}  
\author{Jake A. Soloff ~~~ Adityanand Guntuboyina\thanks{Supported by
  NSF CAREER Grant DMS-16-54589} ~~~ Jim Pitman}  
\date{
	Department of Statistics, University of California, Berkeley \\~\\
	\today
}

\maketitle

\begin{abstract}
It is well known that the isotonic least squares estimator is characterized as the derivative of 
the greatest convex minorant of a random walk. Provided the walk has exchangeable 
increments, we prove that the slopes of the greatest convex minorant are distributed as order 
statistics of the running averages. This result implies an exact non-asymptotic formula for the 
squared error risk of least squares in isotonic regression when the true sequence is constant 
that holds for every exchangeable error distribution.
\end{abstract}

\section{Introduction}\label{sec-intro}

Isotonic regression refers to the problem of estimating a monotone sequence $\theta^*_1\le 
\cdots\le \theta_n^*$ based on a noisy observation vector $Y$ which is assumed to be an
additive perturbation of $\theta^* = (\theta_1^*,\dots, \theta^*_n)$, 
\begin{equation*}
  Y = \theta^* + \sigma Z,
\end{equation*}
where the components $Z_1, \dots, Z_n$ of $Z$ are assumed to have zero
mean and unit variance. It is commonly assumed that $Z_1, \dots, Z_n$
are independent and identically distributed (i.i.d.) but we work with the more general assumption of exchangeability in this paper. A natural estimator for $\theta^*$ in
this setting is the isotonic Least Squares Estimator (LSE), defined as
\begin{equation*}
  \hat{\theta} := \Pi_{\cm^n}(Y) := \argmin_{\theta \in \cm^n} \|Y -
  \theta\|_2^2,
\end{equation*}
where $\|\cdot\|_2$ denotes the usual Euclidean norm on $\R^n$ and 
$\cm^n := \{\theta\in \R^n : \theta_1\le \cdots\le \theta_n\}$ is the monotone cone of length $n$ 
non-decreasing sequences. As $\cm^n$ is a closed convex cone, $\hat{\theta}$ as defined
above exists uniquely; it can also be computed in $O(n)$ time by the pool adjacent
violators algorithm~\cite{brunk1972statistical, grotzinger1984projections}.

The statistical properties of $\hat{\theta}$ are typically studied in terms of the risk or the
normalized mean squared error:  
\begin{equation*}
  R(\hat{\theta}, \theta^*) := \frac{1}{n} \EE_{\theta^*}
  \|\hat{\theta} - \theta^*\|_2^2. 
\end{equation*}
A key quantity in understanding $R(\hat{\theta}, \theta^*)$ is 
\begin{equation*}
  \delta_n(\mu) := \EE_{Z\sim \mu} \|\Pi_{\cm^n}(Z)\|_2^2,
\end{equation*}
where $\mu$ denotes the law of the noise vector $Z$. Indeed, it is clear that 
\[
 \frac{n}{\sigma^2} R(\hat{\theta}, \theta^*) = \delta_n(\mu) ~~~~~~~~~~~~~~~~~~~~\qt{when
 $\theta_1^* = \dots = \theta_n^*$}.
\]
When $\theta_1^*\le \cdots\le \theta_n^*$ are not all equal, let $(A_1,\dots,A_k)$ be the finest
partition of $\{1,\dots,n\}$ such that $\theta^*$ is constant on each $A_i$. It has been shown
~\cite{oymak2016sharp, fang2017risk, bellec2018sharp} that 
\begin{align}\label{eq-sharpmse}
\frac{n}{\sigma^2} R(\hat{\theta},
  \theta^*) 
  \begin{cases} 
   \leq \delta_{n_1}(\mu_{A_1}) + \dots + \delta_{n_k}(\mu_{A_k}) & \text{for
     every } \sigma > 0 \\
   \rightarrow    \delta_{n_1}(\mu_{A_1}) + \dots + \delta_{n_k}(\mu_{A_k}) & \text{as } \sigma \downarrow 0 
  \end{cases},
\end{align}
where $\mu_{A_i}$ denotes the marginal distribution of $(Z_j)_{j\in A_i}$ and $n_i = |A_i|$ is
the length of the $i^\text{th}$ block for all $i=1,\dots,k$. We emphasize that 
\eqref{eq-sharpmse} holds for arbitrarily dependent $Z_1,\dots, Z_n$ with zero mean and finite 
variance. It was also shown in \cite{bellec2018sharp} that $\delta_n(\mu)$ also bounds the risk 
of the isotonic LSE in misspecified settings where $\theta^*$ does not lie in $\mathcal{M}^n$.

The quantity $\delta_n(\mu)$ therefore crucially controls the risk of the isotonic LSE. The goal 
of this paper is to explicitly determine $\delta_n(\mu)$ for every $n \geq 1$ under the additional 
assumption that $Z$ is exchangeable. Specifically, under the assumption of exchangeability, we 
show in Corollary \ref{cor-stat-dim} that, for all $n$,
\begin{equation}\label{om}
  \delta_n(\mu) = \rho n + (1-\rho)H_n,
\end{equation}
where $H_n := 1 + \frac{1}{2} + \dots + \frac{1}{n}$ is the $n^\text{th}$ harmonic number and 
$\rho = \text{Cor}(Z_1,Z_2)$ is the pairwise correlation. Combined with \eqref{eq-sharpmse}, 
our result provides a sharp non-asymptotic bound on the risk of isotonic regression for {\sl any} 
exchangeable noise vector. In the special case when $Z_1,\dots,Z_n$ are i.i.d. with zero mean 
and unit variance, $\rho = 0$ and thus \eqref{om} gives:
\begin{equation}\label{om-ind}
  \delta_n(\otimes_{i=1}^n\eta) = H_n \qt{for every probability
      measure $\eta$}.
\end{equation}
Here $\eta$ is the common distribution of the independent variables $Z_1,\dots,Z_n$.

Previously, the formula~\eqref{om-ind} was known when $\eta$ is the
standard Gaussian probability measure on $\R^n$. This was observed by Amelunxen et 
al.~\cite{amelunxen2014living} who proved it by observing first that when 
$\mu = \otimes_{i=1}^n \eta$ and $\eta$ is the standard Gaussian measure, the formula
\begin{equation}\label{genk}
 \EE \|\Pi_K(Z)\|_2^2 = \sum_{k=0}^n k\,\nu_k(K) 
\end{equation}
holds for every closed convex cone $K \subseteq \R^n$ where
$\nu_k(K)$ is the $k^\text{th}$ intrinsic volume of $K$. When $K = \cm^n$ is the monotone 
cone, the right hand side in equation~\eqref{genk} can be shown to be equal to $H_n$ 
by using the fact that the generating function $s \mapsto \sum_{k=0}^n s^k \nu_k(\cm^n)$ can 
be computed in closed form. Amelunxen et al.~\cite{amelunxen2014living} used the theory of 
finite reflection groups~\cite{coxeter2013generators} to obtain the exact expression for this 
generating function. However, the exact expression for $\sum_{k=0}^n s^k \nu_k(\cm^n)$ can 
already be found in the classical literature on isotonic regression (see Theorem 2.4.2 in 
Roberston et al.~\cite{robertson1988order} and references therein). 

The above proof does not work for non-Gaussian $\eta$ mainly
because the expression~\eqref{genk} does not hold for general $\eta$. In fact,
the best available result on $\delta_n(\otimes_{i=1}^n\eta)$ for non-Gaussian $\eta$ is
in equation (2.11) of Zhang~\cite{zhang2002risk}, who proved the asymptotic result:
\begin{equation*}
  \delta_n(\otimes_{i=1}^n\eta) = (1 + o(1))(1 + \log n) \qt{as $n \rightarrow
    \infty$}. 
\end{equation*}
This bound gives the right behavior as the right hand side of
equation~\eqref{om-ind}  but only as $n \rightarrow \infty$.  We
improve this result by proving for every $n \ge 1$ that $\delta_n(\otimes_{i=1}^n \eta)$ is always 
equal to the $n^\text{th}$ harmonic number $H_n$ for every probability measure $\eta$ having 
mean $0$ and variance $1$.

We prove \eqref{om} by developing a precise characterization of the marginal distribution of 
each individual component $(\Pi_{\mathcal{M}^n}(Z))_k$ of $\Pi_{\mathcal{M}^n}(Z)$. 
Specifically, as long as $Z$ is exchangeable, we show in Theorem
\ref{thm-main} that $(\Pi_{\mathcal{M}^n}(Z))_k$ has the same
distribution as $\bar{Z}_{(k)}$, the  $k^\text{th}$ order statistic of
the running averages $\bar{Z}_j = \frac{Z_1+\dots+Z_j}{j}$. We prove
Theorem \ref{thm-main} in Section \ref{sec-main}, using a
characterization of the components of the isotonic LSE as the
left-hand slopes of the greatest convex minorant of the random walk with increments
$Z_1,\dots,Z_n$. This result and its continuous-time analogue may be
of independent interest outside the study of isotonic regression, so
in Section \ref{sec-main} we also address consequences for the
greatest convex minorant of a stochastic process with exchangeable
increments. The order statistics of the running averages
$\{\bar{Z}_k\}_{k=1}^n$ can be fairly complicated even when $Z$ is
Gaussian; however, Theorem~\ref{thm-main} easily implies results such
as~\eqref{om}. In Section~\ref{sec-consequences}, we detail some risk
calculations for isotonic regression and its variants which all follow
from Theorem~\ref{thm-main}.   

\section{Main Result}\label{sec-main}

Let $S_k=~\sum_{i=1}^k Z_i$ denote the partial sums for $k=1,\dots,n$, started at $S_0 = 0$. 
Identify the random walk $\{S_k\}_{k=0}^n$ with its {\sl cumulative sum diagram} (CSD) 
$S : [0,n]\to\R$, where $S(k) = S_k$ for integers $k=0,\dots,n$ and linearly interpolated 
between integers. Let $C : [0,n]\to\R$ denote the {\sl greatest convex minorant} (GCM) of $S$, 
i.e. the greatest convex function that lies below $S$. See Figure~\ref{fig-gcm} for a depiction of 
the GCM of the CSD. With this notation, we now recall the graphical representation of the 
isotonic LSE as given in Theorem 1.2.1 of Roberston et al.~\cite{robertson1988order}.

\begin{lemma}\label{lem-characterization} For any vector $Z$, the isotonic LSE 
$\Pi_{\cm^n}(Z)$ is given by the left-hand slopes of the greatest convex minorant of the 
cumulative sum diagram. For all $k=1,\dots,n$
\[
\left(\Pi_{\cm^n}(Z)\right)_k  = C(k) - C(k-1)= \partial_- C(k).
\]
\end{lemma}

\begin{figure}[h!]
\centering
\includegraphics[scale=.25]{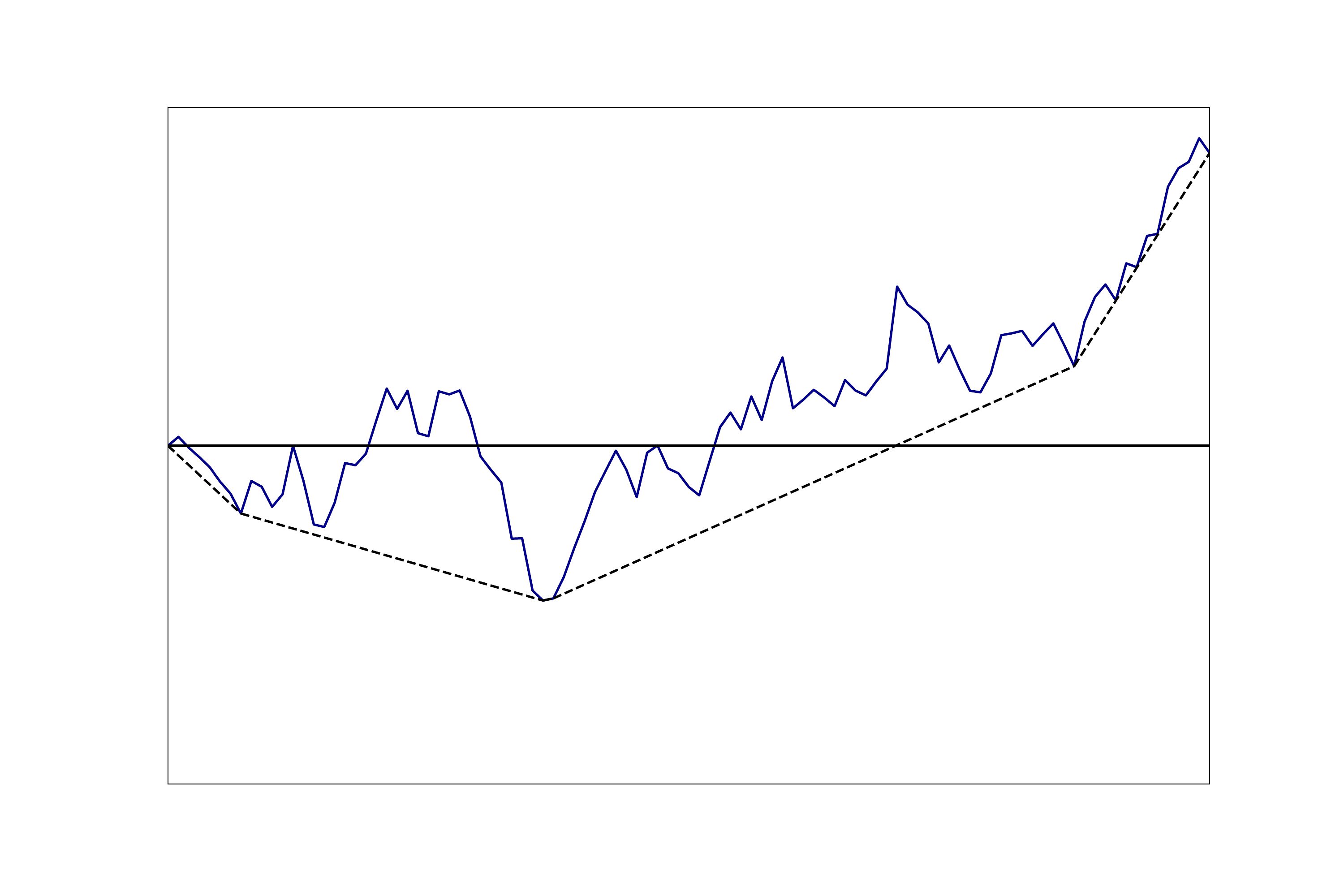}
\caption{Solid blue curve is the cumulative sum diagram $S$ of increments $Z_1,\dots,Z_n$; 
dashed black curve is the greatest convex minorant $C$ of $S$.}\label{fig-gcm}
\end{figure}

For the remainder of this section let 
\begin{align}
\Delta_k := \partial_- C(k)= \min_{k\le v \le n} \max_{0\le u < k} \frac{S_v - S_{u}}{v - u}
\end{align}  
denote the left-hand slope of the GCM at $k$, so $\Delta = (\Delta_1,\dots, \Delta_n)$ is equal 
to $\Pi_{\cm^n}(Z)$ by the lemma. In particular, when $k=1$ we have $\Delta_1 = \min_{1\le 
v\le n}\frac{S_v}{v}$. When $k=n$, we have $\Delta_n = \max_{0\le u<n}\frac{S_n-S_u}{n-u}$, 
and if $(Z_n,\dots, Z_1)\ed (Z_1,\dots, Z_n)$ then 
$\Delta_n\ed  \max_{1\le u\le n}\frac{S_u}{u}$. Our next result generalizes this observation, 
showing that the $k^\text{th}$ slope $\Delta_k$ is equal in distribution to the $k^\text{th}$ 
smallest running average if $Z$ is exchangeable.

\begin{theorem}\label{thm-main} Suppose $Z = (Z_1,\dots,Z_n)$ is exchangeable. Let 
$\bar{Z}_k := \frac{1}{k}\sum_{i=1}^k Z_i$ denote the $k^\text{th}$ running average for 
$k=1,\dots,n$ and let $\bar{Z}_{(1)}\le\cdots \le \bar{Z}_{(n)}$ denote their order statistics. Then
\begin{align}\label{eq-dist-identity}
\Delta_k
\stackrel{d}{=} \bar{Z}_{(k)}
\end{align}
marginally for all $k=1,\ldots,n$. 
\end{theorem}

\begin{proof} As before, let $S_k$ denote the $k^\text{th}$ partial sum. Let $M$ be the last 
argmin of the sequence $\{S_i\}_{i=0}^n$, and let $N$ be the amount of time the walk is non-positive
$N:= \sum_{i = 1}^n 1 ( S_i \le 0 )$. We will use Corollary 11.14 of Kallenberg 
\cite{kallenberg2006foundations}, due to Sparre-Andersen, which says $M \ed N$ as long as 
$Z$ is exchangeable.

Note that the slope of the GCM switches from non-positive to positive at time $M$, since the 
horizontal line with intercept $S_M$ minorizes the GCM and touches it at time $M$. Hence, no 
matter the sequence of increments $Z_i$, there is the identity of events
\begin{align}\label{eq-min-id}
(\Delta_k \le 0 ) =  (M \ge k ).
\end{align}
Also, for the time $N$ that the walk is non-positive, since $S_i \le 0$ if and only if 
$\bar{Z}_i\le 0$, there is the identity of events 
$$
(\bar{Z}_{(k)} \le 0 ) =  (N \ge k ) .
$$
The equality in distribution $M \ed N$ then implies
$$
\mathbb{P}( \Delta_k \le 0 ) =  \mathbb{P}( \bar{Z}_{(k)} \le 0 ).
$$
If the sequence $\{Z_i\}$ is modified to $\{Z_i - z\}$ for some fixed $z$, the modified sequence 
is exchangeable, and the values of $\Delta_k$ and $\bar{Z}_{(k)}$ for the modified sequence 
are just $\Delta_k - z$ and $\bar{Z}_{(k)} - z$. 
Applying the above identity to the modified sequence gives
$$
\mathbb{P}( \Delta_k \le z ) =  \mathbb{P}( \Delta_k - z \le 0 ) =  \mathbb{P}( \bar{Z}_{(k)} - z \le 0 )  = \mathbb{P}( \bar{Z}_{(k)} \le z).
$$
So $\Delta_k$ and $\bar{Z}_{(k)}$ have the same cumulative distribution function, hence the 
same distribution.  
\end{proof}

The proof of Theorem \ref{thm-main} has a straightforward generalization to the setting where 
$S:[0,1]\to\R$ is a continuous-time stochastic process. Knight \cite{knight1996uniform} showed 
that the analogous distributional identity $M\ed N$ holds when $S$ has exchangeable 
increments and $S(0) = 0$. Hence, by a similar proof, we find that the slope $\Delta(p)$ of the 
greatest convex minorant of $S$ at time $p\in [0,1]$ has the same distribution as the 
$p^\text{th}$ percentile point of the occupation measure for the process 
$(\frac{S(t)}{t}, 0\le t\le 1)$. We record this result as the following corollary.

\begin{corollary}\label{cor-cts-time} Let $S$ denote a real-valued c\`adl\`ag stochastic process 
on $[0,1]$ with exchangeable increments, such that $S(0) = 0$. Define $\Delta(t)$ as the slope 
of the greatest convex minorant of $S$ at $t$, and let $F:\R\to [0,1]$ denote the (random) cdf 
associated with the occupation measure of $(\frac{S(t)}{t}, 0\le t\le 1)$,
\begin{align}\label{eq-occ-cdf}
F(x) = \lambda(\{t\in [0,1] : S(t)\le tx\}), 
\end{align}
 where $\lambda$ denotes Lebesgue measure. Then 
\begin{align}
\Delta(p) 
= \inf_{p\le v\le 1}\sup_{0\le u< p}\frac{S(v)-S(u)}{v-u}
\ed F^{-1}(p)
\end{align}
marginally for all $p\in [0,1]$. 
\end{corollary}

See Abramson et al. \cite{abramson2011convex} for a general study of convex minorants of 
random walks and processes with exchangeable increments. In the special cases where $S$ is 
a standard Brownian motion or Brownian bridge on the unit interval, Carolan \& Dykstra 
\cite{MR1891744} derive the distribution of the slope $\Delta(p)$, jointly with the process $S(p)$ 
and its convex minorant at $p$, for a fixed value $p\in [0,1]$. Given our corollary, their explicit 
formula for the slope $\Delta(p)$ provides the distribution of $F^{-1}(p)$, giving new information 
about the occupation measure of $(\frac{S(t)}{t}, 0\le t\le 1)$ for Brownian motion and Brownian 
bridge. The distribution of the $p^\text{th}$ percentile point of the occupation measure for 
$(S(t), 0\le t\le 1)$ has been obtained under the same generality as Corollary \ref{cor-cts-time}: 
see the introduction of Dassios \cite{dassios2005quantiles} and references therein.

\section{Consequences for Isotonic Regression}\label{sec-consequences}

Since the identity of Theorem \ref{thm-main} holds marginally, it allows us to simplify 
expectations of functions that are additive in the components of $\Pi_{\cm^n}(Z)$. 
As long as $Z$ is exchangeable,
\begin{align}\label{eq-general-corollary}
\sum_{k=1}^n\EE h((\Pi_{\cm^n}(Z))_k)
=\sum_{k=1}^n\EE  h(\bar{Z}_{(k)})
=\sum_{k=1}^n\EE  h(\bar{Z}_{k}).
\end{align}
Taking $h(x) = |x|^p$, we obtain our first corollary.

\begin{corollary}
Suppose $Z = (Z_1,\dots,Z_n)$ is exchangeable. For $p > 0$,
\begin{align}\label{eq-thm}
\EE\|\Pi_{\cm^n}(Z)\|_p^p
&= \sum_{k=1}^n \EE\left|\frac{1}{k}\sum_{i=1}^k Z_i\right|^p,
\end{align}
provided $\EE|Z_1|^p < \infty$.
\end{corollary}

\begin{remark}
Viewed through its graphical representation, $\Delta_k = C(k) -C(k-1)$ is the left-derivative of 
the GCM $C$ at $k$, so when the power $p=1$, equation \eqref{eq-thm} yields the discrete 
arc-length formula
\begin{align}
\sum_{k=1}^n\EE |C(k)-C(k-1)| 
&= \EE\|\Pi_{\cm^n}(Z)\|_1
= \sum_{k=1}^n \frac{1}{k}\EE|S_k|
\end{align}
Closely related to this formula is the identity of Spitzer \& 
Widom~\cite{spitzer1961circumference}, which takes $\tilde{Z}_1,\dots,\tilde{Z}_n$ to be a 
sequence of i.i.d. random variables in $\R^2$ (or the complex plane $\mathbb{C}$) with finite 
variance. If $\tilde{S}_k = \sum_{i=1}^k \tilde{Z}_i$ is the partial sum and $\tilde{L}_n$ is the 
length of the perimeter of the convex hull $\text{conv}(0, \tilde{S}_1,\dots, \tilde{S}_n)$, then
\begin{align}
\EE \tilde{L}_n
&= 2\sum_{k=1}^n \frac{1}{k}\EE\|\tilde{S}_k\|.
\end{align}
These formulas connect the geometry of the convex hull of a random walk to the magnitudes 
of the running means. 
\end{remark}

Consider the case when $p = 2$. Since $Z$ is exchangeable, every pair of components has the 
same correlation $\rho$. If we further assume $Z_1$ has zero mean and unit variance, the right 
hand side of equation~\eqref{eq-thm} can be computed explicitly 
\begin{align*}
\EE\left(\frac{1}{k}\sum_{i=1}^kZ_i\right)^2
=\rho + \frac{1-\rho}{k}.
\end{align*}
Summing over $k$ yields our next result.
\begin{corollary}\label{cor-stat-dim} Suppose $Z\sim \mu$ is an exchangeable random vector 
with zero mean, unit variance, and pairwise correlation~$\rho$. Then
\begin{align*}
\delta_n(\mu)
 = \rho n + (1-\rho)H_n.
\end{align*}
\end{corollary}

This result should be contrasted with other distribution-free identities, namely
\begin{align*}
\EE\|Z\|_2^2 = n
\text{ and }
\EE\|\bar{Z}_n{\bf 1}_n\|_2^2 = 1,
\end{align*}
provided $Z$ has i.i.d. components with zero mean and unit variance. In particular, suppose we 
observe $Y = \theta^* + \sigma Z$ where $Z$ has i.i.d. components with zero mean and unit 
variance, but it turns out that $\theta^* = c{\bf 1}_n$ is constant. 
If we know $\theta^*$ is constant, we can estimate it by a constant
sequence $\bar{Y}{\bf 1}_n$ and pay a constant price $\sigma^2$ in
total risk. If we know nothing about the structure of $\theta^*$ and
use $\hat\theta = Y$, the risk $\sigma^2 n$ is quite large by
comparison. The monotone sequence estimate resides in the middle, with
a much smaller risk of $\sigma^2H_n$ and knowledge only about the relative
order. We explained in Section~\ref{sec-intro} how risk calculations
when $\theta^* = 0$ generalize to MSE bounds that are sharp in the low
noise limit for arbitrary $\theta^*$. For example, when
  $\theta^* \in \mathcal{M}^n$ has $k$ constant pieces, then
  \eqref{eq-sharpmse}, Corollary \ref{cor-stat-dim} and the fact that
  $H_l \leq \log(el)$ for every $l \geq 1$ imply that  
  \begin{equation*}
    R(\hat{\theta}, \theta^*) \leq \frac{k\sigma^2}{n}
    \log\left(\frac{en}{k} \right)
  \end{equation*}
  whenever $Z_1, \dots, Z_n$ are i.i.d. with mean zero and unit
  variance. Also if $\theta^* \in \R^n$ is not necessarily in
  $\mathcal{M}^n$, then Corollary \ref{cor-stat-dim}, together with
  the results of \cite{bellec2018sharp}, implies that
  \begin{equation*}
    R(\hat{\theta}, \theta^*) \leq \inf_{\theta \in \mathcal{M}^n}
    \left(\frac{1}{n} \|\theta - \theta^*\|^2 + \frac{\sigma^2 k(\theta)}{n}
    \log\left(\frac{en}{k(\theta)} \right) \right),
  \end{equation*}
where $k(\theta)$ is the number of constant pieces of the vector $\theta$. 
These formulae (with the leading constant of 1 in front of the
$\frac{k\sigma^2}{n} \log \frac{en}{k}$ term on the right
  hand side) were previously only known when the distribution of $Z_1,
  \dots, Z_n$ was standard Gaussian.

Define the $L^p$-risk of the isotonic LSE
\begin{align*}
R^{(p)}(\hat\theta,\theta^*)
=\frac{1}{n}\EE\|\hat\theta-\theta^*\|_p^p
\end{align*}
so that $R(\hat\theta,\theta^*) = R^{(2)}(\hat\theta,\theta^*)$. We can similarly employ 
Theorem~\ref{thm-main} to explicitly calculate the $L^p$-risk of the isotonic LSE $\hat\theta$ 
when $\theta^*$ is constant and $Z$ is Gaussian:

\begin{corollary}\label{cor-lp} Suppose $Z\sim\cn(0, I_n)$. Then for any $p> 0$,
$$
\EE\|\Pi_{\cm^n}(Z)\|_p^p 
= H_{n, p/2}\EE|Z_1|^p
= H_{n, p/2}\sqrt{\frac{2^{p}}{\pi}}\Gamma\left(\frac{p+1}{2}\right),
$$
where $H_{n,m} = \sum_{k=1}^n \frac{1}{k^m}$.
\end{corollary}
\begin{proof} Note 
$\EE\left|\frac{1}{k}\sum_{i=1}^kZ_i\right|^p = \left(\frac{2}{k}\right)^{p/2}\frac{\Gamma\left(\frac{p+1}{2}\right)}{\sqrt{\pi}}$ 
and apply the theorem.
\end{proof}

Corollary \ref{cor-lp} should similarly be contrasted with the following identities when 
$Z\sim \cn(0, I_n)$ :
$$
\EE\|Z\|_p^p = n\EE|Z_1|^p \text{ and } \EE\|\bar{Z}{\bf 1}_n\|_p^p = n^{1-p/2}\EE|Z_1|^p
$$
respectively. In particular, when $p > 2$, the bound 
$H_{n, p/2} < \sum_{k=1}^\infty \frac{1}{k^{p/2}} < \infty$ holds for all $n$, which is to say 
$\EE\|\Pi_{\cm^n}(Z)\|_p^p$ is bounded when $p > 2$ whereas $\EE\|Z\|_p^p$ grows without 
bound as $n$ grows. 

When $\theta^*$ is constant and $Z\sim \cn(0, I_n)$, the $L^p$ risk of isotonic regression is
\begin{align}\label{eq-gaussian-lp}
R^{(p)}(\hat\theta,\theta^*)
&=\frac{H_{n,p/2}}{n}\sigma^p\EE|Z_1|^p .
\end{align}
When $1\le p\le 2$, Theorem 2.3 of Zhang~\cite{zhang2002risk} shows an asymptotic result for 
the $L^p$ risk on constant $\theta^*$ that agrees with equation \eqref{eq-gaussian-lp}. 

The continuous-time distributional identity in Corollary \ref{cor-cts-time} applies to the 
asymptotic distribution of the isotonic least squares estimator. A standard model for studying the 
asymptotic behavior of isotonic regression is
\begin{align*}
\theta^*_k &= f^*\left(\frac{k}{n}\right) 
\end{align*}
where $f^* : [0,1]\to \R$ is non-decreasing. We observe $Y$, a noisy version of $\theta^*$, and 
calculate $\hat\theta$ by projecting $Y$ onto the monotone cone. The function estimate $\hat{f}$
 is defined by $\hat{f}\left(\frac{k}{n}\right) = \hat\theta_k$ and linearly interpolated between 
 design points. Here, as before, the dependence on $n$ in $\theta^*\in \cm^n$ is suppressed, 
 but now we are interested in the behavior of isotonic least squares $\hat{f}(p)$ at a fixed point 
 $p\in[0,1]$ as $n\to \infty$. 

Define the partial sum process $S^{(n)} : [0,1]\to \R$ by 
$S^{(n)}(k/n) = \frac{Y_1+\cdots+Y_k}{\sqrt{n}}$, linearly interpolated between design points. 
When the function $f^*\threequals c$ is constant, the quantity $$\sqrt{n}(\hat{f}(p)-f^*(p))$$  is 
given by the left-derivative of the greatest convex minorant of $S^{(n)}$ at $p$. By the 
invariance principle, this converges in distribution to the left-derivative of the greatest convex 
minorant of standard Brownian motion $B = (B(t), 0\le t\le 1)$ at $t_0$. This asymptotic result is 
well known and a similar result was noted for the Grenander estimator in Carolan \& Dykstra 
\cite{carolan1999asymptotic}, where Brownian motion is replaced with a Brownian bridge. 
Corollary \ref{cor-cts-time} relates this asymptotic distribution to the percentile points of the 
occupation measure for $(\frac{B(t)}{t}, 0\le t\le 1)$. 

Finally, Corollary \ref{cor-stat-dim} on the projection onto $\cm^n$ extends over to that of the 
set of non-negative monotone sequences $\cm^n_+ = \cm^n\cap \R^n_+$. Theorem 1 of 
N\'emeth \& N\'emeth \cite{nemeth2012project} observes that the projection of $Z$ onto 
$\cm_+^n$ is given by $\Pi_{\cm_+^n}(Z) = \Pi_{\cm^n}(Z)_+$, the element-wise positive part of 
the projection onto $\cm^n$. Hence the distributional identity Theorem \ref{thm-main} yields a 
similar set of identities for non-negative isotonic regression.
\begin{corollary}
For any exchangeable noise vector $Z$, 
\begin{align}
(\Pi_{\cm^n_+}(Z))_k
\ed (\bar{Z}_{(k)})_+
\end{align}
Provided $\EE|Z_i|^p < \infty$,
\begin{align}\label{eq-nnmonotone}
\EE\|\Pi_{\cm^n_+}(Z)\|_p^p
&= \sum_{k=1}^n \EE\left(\frac{1}{k}\sum_{i=1}^k Z_i\right)_+^p,
\end{align}
Furthermore, if $Z$ is symmetric with unit variance, the generalized statistical dimension of the 
monotone cone is
\begin{align}\label{eq-nnmonotone-dim}
\EE\|\Pi_{\cm^n_+}(Z)\|_2^2 = \frac{\rho n + (1-\rho)H_n}{2},
\end{align}
where $\rho$ is the pairwise correlation.
\end{corollary}
\begin{proof} Equation~\eqref{eq-nnmonotone} follows from equation 
\eqref{eq-general-corollary} by taking $h(x) = (x)_+^p$. When $Z_i\stackrel{d}{=}-Z_i$ is 
symmetric with unit variance,
\begin{align*}
\EE\left(\frac{1}{k}\sum_{i=1}^kZ_i\right)_+^2 
=\frac{1}{2}\EE\left(\frac{1}{k}\sum_{i=1}^kZ_i\right)^2 
= \frac{1}{2}\left(\rho + \frac{1-\rho}{k}\right).
\end{align*}
Summing over $k$ yields equation~\eqref{eq-nnmonotone-dim}.
\end{proof}
Equation~\eqref{eq-nnmonotone-dim} is also shown in Amelunxen et 
al.~\cite{amelunxen2014living} in the special case $Z\sim \cn(0, I_n)$ using the theory of finite 
reflection groups. The identity~\eqref{eq-nnmonotone} allows us to show 
equation~\eqref{eq-nnmonotone-dim} for a much wider variety of noise vectors, and as before 
also allows us to obtain relations for the expected $L^p$ norms of the projection of the noise 
vector. All of our exact formulae follow from the distributional identity in Theorem \ref{thm-main}, 
which exploits the geometric characterization of the isotonic LSE in Lemma 
\ref{lem-characterization}. An interesting open question is whether similar characterizations---
such as for convex regression \cite{groeneboom2001estimation}---may yield exact non-asymptotic 
risk calculations in other shape-constrained estimation problems. 

\bibliographystyle{plain}
\bibliography{main.bbl}

\end{document}